\documentclass[11pt,a4paper]{amsart}
\usepackage[margin=25mm]{geometry}
\usepackage[numbers]{natbib}
\usepackage{hyperref}

\usepackage{amssymb,amsmath}
\usepackage{amsthm}
\usepackage{bbm, bm}
\usepackage{stmaryrd}
\usepackage[usenames,dvipsnames]{xcolor}
\usepackage{color}
\usepackage{graphicx}
\usepackage{caption}
\usepackage{float}
\usepackage{wrapfig}
\usepackage{enumerate}  

\usepackage{lineno}
\usepackage{abstract}

\input xy
\xyoption{all} %\tolerance=500

\numberwithin{equation}{section}

\newtheorem{theorem}{Theorem}[section]

\newtheorem{lemma}[theorem]{Lemma}
\newtheorem{proposition}[theorem]{Proposition}

\theoremstyle{definition}
\newtheorem{definition}[theorem]{Definition}
\newtheorem{example}[theorem]{Example}
\newtheorem{coexample}[theorem]{Counter Example}

\newenvironment{remark}[1][Remark.]{\begin{trivlist}
\item[\hskip \labelsep {\bfseries #1}]  }{ \end{trivlist}}

\newcommand{\rmd}{\textnormal{d}}
\newcommand{\rme}{\textnormal{e}}
\newcommand{\rmh}{\textnormal{h}}

\DeclareMathOperator{\Vect}{Vect}

\DeclareMathOperator{\Span}{Span}

\DeclareMathOperator{\Image}{Im}

\DeclareMathOperator{\tr}{tr}
 \DeclareMathOperator{\sgn}{sgn}

\font\black=cmbx10 \font\sblack=cmbx7 \font\ssblack=cmbx5 \font\blackital=cmmib10  \skewchar\blackital='177
\font\sblackital=cmmib7 \skewchar\sblackital='177 \font\ssblackital=cmmib5 \skewchar\ssblackital='177
\font\sanss=cmss10 \font\ssanss=cmss8 %scaled 900
\font\sssanss=cmss8 scaled 600 \font\blackboard=msbm10 \font\sblackboard=msbm7 \font\ssblackboard=msbm5
\font\caligr=eusm10 \font\scaligr=eusm7 \font\sscaligr=eusm5  \font\fraktur=eufm10
\font\sfraktur=eufm7 \font\ssfraktur=eufm5 
\font\bsymb=cmsy10 scaled\magstep2
\def\all#1{\setbox0=\hbox{\lower1.5pt\hbox{\bsymb
       \char"38}}\setbox1=\hbox{$_{#1}$} \box0\lower2pt\box1\;}
\def\exi#1{\setbox0=\hbox{\lower1.5pt\hbox{\bsymb \char"39}}
       \setbox1=\hbox{$_{#1}$} \box0\lower2pt\box1\;}

\def\tx#1{{\fam0\relax#1}}

\newfam\bifam
\textfont\bifam=\blackital \scriptfont\bifam=\sblackital \scriptscriptfont\bifam=\ssblackital

\newfam\blfam
\textfont\blfam=\black \scriptfont\blfam=\sblack \scriptscriptfont\blfam=\ssblack

\newfam\bbfam
\textfont\bbfam=\blackboard \scriptfont\bbfam=\sblackboard \scriptscriptfont\bbfam=\ssblackboard

\newfam\ssfam
\textfont\ssfam=\sanss \scriptfont\ssfam=\ssanss \scriptscriptfont\ssfam=\sssanss
\def\sss#1{{\fam\ssfam\relax#1}}

\newfam\clfam
\textfont\clfam=\caligr \scriptfont\clfam=\scaligr \scriptscriptfont\clfam=\sscaligr

\newfam\frfam
\textfont\frfam=\fraktur \scriptfont\frfam=\sfraktur \scriptscriptfont\frfam=\ssfraktur

\def\hpb#1{\setbox0=\hbox{${#1}$}
    \copy0 \kern-\wd0 \kern.2pt \box0}
\def\vpb#1{\setbox0=\hbox{${#1}$}
    \copy0 \kern-\wd0 \raise.08pt \box0}

\def\pmb#1{\setbox0\hbox{${#1}$} \copy0 \kern-\wd0 \kern.2pt \box0}
\def\pmbb#1{\setbox0\hbox{${#1}$} \copy0 \kern-\wd0
      \kern.2pt \copy0 \kern-\wd0 \kern.2pt \box0}
\def\pmbbb#1{\setbox0\hbox{${#1}$} \copy0 \kern-\wd0
      \kern.2pt \copy0 \kern-\wd0 \kern.2pt
    \copy0 \kern-\wd0 \kern.2pt \box0}
\def\pmxb#1{\setbox0\hbox{${#1}$} \copy0 \kern-\wd0
      \kern.2pt \copy0 \kern-\wd0 \kern.2pt
      \copy0 \kern-\wd0 \kern.2pt \copy0 \kern-\wd0 \kern.2pt \box0}
\def\pmxbb#1{\setbox0\hbox{${#1}$} \copy0 \kern-\wd0 \kern.2pt
      \copy0 \kern-\wd0 \kern.2pt
      \copy0 \kern-\wd0 \kern.2pt \copy0 \kern-\wd0 \kern.2pt
      \copy0 \kern-\wd0 \kern.2pt \box0}

\mathchardef\za="710B  %\alpha
\mathchardef\zb="710C  %\beta
\mathchardef\zg="710D  %\gamma
\mathchardef\zd="710E  %\delta
\mathchardef\zve="710F %\epsilon
\mathchardef\zz="7110  %\zeta
\mathchardef\zh="7111  %\eta
\mathchardef\zvy="7112 %\theta
\mathchardef\zi="7113  %\iota
\mathchardef\zk="7114  %\kappa
\mathchardef\zl="7115  %\lambda
\mathchardef\zm="7116  %\mu
\mathchardef\zn="7117  %\nu
\mathchardef\zx="7118  %\xi
\mathchardef\zp="7119  %\pi
\mathchardef\zr="711A  %\rho
\mathchardef\zs="711B  %\sigma
\mathchardef\zt="711C  %\tau
\mathchardef\zu="711D  %\upsilon
\mathchardef\zvf="711E %\phi
\mathchardef\zq="711F  %\chi
\mathchardef\zc="7120  %\psi
\mathchardef\zw="7121  %\omega
\mathchardef\ze="7122  %\varepsilon
\mathchardef\zy="7123  %\vartheta
\mathchardef\zf="7124  %\varomega
\mathchardef\zvr="7125 %\varrho
\mathchardef\zvs="7126 %\varsigma
\mathchardef\zf="7127  %\varphi
\mathchardef\zG="7000  %\Gamma
\mathchardef\zD="7001  %\Delta
\mathchardef\zY="7002  %\Theta
\mathchardef\zL="7003  %\Lambda
\mathchardef\zX="7004  %\Xi
\mathchardef\zP="7005  %\Pi
\mathchardef\zS="7006  %\Sigma
\mathchardef\zU="7007  %\Upsilon
\mathchardef\zF="7008  %\Phi
\mathchardef\zW="700A  %\Omega
\mathchardef\zC="7009  %\Psi

\newcommand{\be}{\begin{equation}}
\newcommand{\ee}{\end{equation}}

\newcommand{\bea}{\begin{eqnarray}}
\newcommand{\eea}{\end{eqnarray}}
\def\*{{\textstyle *}}
\newcommand{\R}{{\mathbb R}}

\newcommand{\Z}{{\mathbb Z}}

\newcommand{\s}{{\textstyle *}}

\def\Sec{\sss{Sec}}
\def\Vect{\sss{Vect}}

\def\sH{{\sss H}}

\def\sT{{\sss T}}

\def\sV{{\sss V}}

\def\xi{\tx{i}}

\def\s*{{\scriptstyle *}}

\newcommand{\beas}{\begin{eqnarray*}}
\newcommand{\eeas}{\end{eqnarray*}}
\def\half{\frac{1}{2}}

\title{Carrollian $\R^\times$-bundles: Connections and Beyond} 
\author{Andrew James Bruce } 
   \email{andrewjamesbruce@googlemail.com}

   \date{\today}
 
\begin{document}
 %%%%%%%%%%%%%%%%%%%%%%%%%%%
 \maketitle
%%%%%%%%%%%%%%%%%%%%%%%%%%%
\vspace{-20pt}
%%%%%%%%%%%%%%%%%%%%
\begin{abstract}{\noindent  We propose an approach to Carrollian geometry using principal $\R^\times$-bundles ($\R^\times := \R \setminus \{0\}$) equipped with a degenerate metric whose kernel is the module of vertical vector fields. The constructions allow for non-trivial bundles, and a large class of Carrollian manifolds can be analysed in this formalism. A key result in this is that once a principal connection has been selected, there is a canonical non-degenerate metric that can be leveraged to circumvent the difficulties associated with a degenerate metric. Within this framework, we examine the Levi-Civita connection and null geodesics.}\\

\noindent {\Small \textbf{Keywords:} Carrollian Geometry;~Principal Bundles}

\end{abstract}
\tableofcontents
\section{Introduction}
Carrollian manifolds were introduced by Duval et al. \cite{Duval:2014a,Duval:2014b, Duval:2014} (earlier related works include  Lévy-Leblond \cite{Lévy-Leblond:1965}, Sen Gupta \cite{SenGupta:1966}, and Heanneaux \cite{Heanneaux:1979}) as manifolds equipped with a degenerate metric whose kernel is spanned by a nowhere vanishing complete vector field. Natural examples of Carrollian manifolds include null hypersurfaces, such as punctured future or past light-cones in Minkowski spacetime, and the event horizon of a Schwarzschild black hole.  Applications in theoretical physics of Carrollian geometry include the study of black hole horizons (see \cite{Donnay:2019}), boundaries of asymptotically flat spacetimes (see \cite{Duval:2014b}), relativistic strings in the limit $c  \rightarrow 0$ (see \cite{Cardona:2016}), and cosmology (see \cite{deBoer:2022}), to name a few. A review of a wide range of applications of Carrollian physics can be found in \cite{Bagchi:2025}. The nomenclature ``Carrollian" was introduced by Lévy-Leblond, who noticed that in the ultra-relativistic limit, massive particles move at infinite speed and yet remain stationary. This is somewhat similar to the situation described in the dialogue between the Red Queen and Alice in Lewis Carroll's \emph{Through the Looking Glass}.  The picture to keep in mind is that in the ultra-relativistic limit ($c\rightarrow 0$), light cones in Minkowski spacetime collapse to lines in the time direction for every observer. Causality and dynamics are lost, but remarkably, theories in this limit are still useful.  \par 
In this paper, we present a novel approach to Carrollian geometry using principal $\R^\times$-bundles\footnote{Here we set $\R^\times := \R \setminus \{0\}$ with the group operation being standard multiplication.} equipped with a degenerate metric whose kernel is the module of vertical vector fields (see Definition \ref{def:CarrBun}). As the vertical bundle is one-dimensional, we view the metric as being minimally degenerate. To motivate our use of principal $\R^\times$-bundles, consider the punctured light cone in Minkowski spacetime $C \cong S^2 \times \R^\times$.  The half cone was discussed by Duval et al. \cite{Duval:2014b}. We can describe the points on the cone as $(t\, \mathbf{u}, t)\in \R^{1,3}$, with $\mathbf{u} \in S^2$ and $t \in \R^\times$. The projection we define as $(t\, \mathbf{u}, t)\mapsto \mathbf{u}$. The $\R^\times$-action is multiplication of the fibre coordinate by a non-zero real number. The fundamental vector field of this action, which we will refer to as the Euler vector field, is $\Delta = t \partial_t$. The degenerate metric we define as $g = t^2\,g_{S^2}$, where $g_{S^2}$ is the round metric on $S^2$. Note that $\ker(g)$ is spanned by the Euler vector field. Our notion of a Carrollian $\R^\times$-bundle will closely mimic this natural example. \par 
We prove that Carrollian manifolds that have a simple foliation by $\R$ induced by the kernel of the degenerate metric (we call these simple Carrollian manifolds, see Definition \ref{def:CarrMan}) can, non-canonically, be associated with Carrollian $\R^\times$-bundles (see Theorem \ref{thrm:CarrManAreCarrBun}). Thus, our formalism covers a physically interesting class of Carrollian manifolds, including those whose temporal fibrations lead to non-trivial fibre bundles. The philosophy is that working with the principal bundle picture can simplify the mathematics and, with care, shed light on Carrollian geometry.  \par
Other results include:
\begin{itemize}
\item If the fundamental vector field of the $\R^\times$-action is Killing, then there is an induced (non-degenerate) metric on the base manifold - Proposition \ref{prop:KillingEulerMetric}.
\item If the fundamental vector field is Killing, then all vertical vector fields are Killing - Proposition \ref{prop:Killing}. 
\item Once an $\R^\times$-connection has been chosen, then there is (canonically) an affine connection on the principal bundle that is torsionless, but not, in general, metric compatible - Proposition \ref{prop:TorsionlessCon}.
\item  The constructed affine connection on the $\R^\times$-principal bundle extends to an affine connection on the associated line bundle, provided there is a regularity condition on the vector fields - Proposition \ref{prop:ExtLCcon}.
\end{itemize}
We remark that underlying the existence of an affine connection from an $\R^\times$-connection is a Kaluza--Klein geometry. The construction is to build a non-degenerate metric from the degenerate metric and a connection, and then to consider the associated Levi-Civita connection. Within the context of standard Carrollian geometry, Blitz and McNut \cite{Blitz:2024} argue that affine connections on a Carrollian manifold that are metric and compatible with the fundamental vector field, but may have torsion, are natural. However, torsionless connections that are not metric-compatible have appeared in the literature. For example, Chandrasekaran et al. \cite{Chandrasekaran:2022} find, in the context of the null Brown-York stress tensor and related conservation laws, that torsionless and, generally, non-metric connections naturally appear. It is well known that Carrollian manifolds do not come equipped with a canonical affine connection akin to the Levi-Civita connection of Riemannian geometry. Connections have to be proposed as extra data based on their required properties.  The approach we propose does not fully resolve this issue; principal connections are not, in general, unique. However, on a trivial $\R^\times$-bundle, there is always the trivial connection; this canonical choice will simplify some physically interesting examples.
\medskip 

\noindent \textbf{Looking Forward:}  The framework of Carrollian $\R^\times$-bundles equipped with a connection allows the machinery of principal bundles and (pseudo-)Riemannian geometry to be applied to Carrollian geometry. Recall that a \emph{Carrollian bundle} (see \cite{Ciambelli:2019}) is a triple $(E, g, \kappa)$, where $\pi :E\rightarrow M$ is a fibre bundle with typical fibre $\R$, $g$ is a degenerate metric on $E$ of signature $(1,1, \cdots, 1, 0)$ (the zero is in the fibre position), and $\kappa \in \Vect(E)$ is a complete and non-singular Killing vector field, such that $\ker(g) = \Span \{\kappa\}$. The selection of a section of $E$ allows a linearisation of $E$, which is a line bundle $L$ over $M$. Moreover, there is a fibre-preserving diffeomorphism that can be composed with the smooth inclusion $P \hookrightarrow L$ (see the proof of Theorem \ref{thrm:CarrManAreCarrBun} for details)
$$P \hookrightarrow L \stackrel{\sim}{\rightarrow} E\,,$$
where  $P = L^\times = L\setminus \{0_M\}$ is the associated $\R^\times$-principal bundle. 
The bundle $P$ can then be equipped with the structure of a Carrollian $\R^\times$-bundle (see Definition \ref{def:CarrBun}) induced by the Carrollian structure on $E$. For example, the Riemannian structure associated with a principal connection is not well-defined on $L$; the connection contains a factor of $ t^{-1}$. Nonetheless, one can work on the Carrollian $\R^\times$-bundle, and then carefully examine if the notions/constructions survive the smooth inclusion. The author has obtained results on constructing intrinsic sigma models on a Carrollian manifold (see \cite{Bruce:2025a}), and with defining the Hodge--de Rham Laplacian in this setting (see \cite{Bruce:2025b}). Proposition \ref{prop:ExtLCcon} on expanding the affine connection from $P$ to $L$ is part of this philosophy of working with principal bundles as the primary mathematical objects. 
\medskip

\noindent \textbf{Arrangement:} In Section \ref{sec:Prelims}, we recall the required theory of $\R^\times$-principal bundles, including $\R^\times$-connections, as needed later in this paper. In Section \ref{sec:CarrGeom}, we present the core ideas of this paper. In Subsection \ref{subsec:CarRBundles}, the notion of a Carrollian $\R^\times$-bundle is carefully presented, including some simple examples. The vital theorem relating Carrollian manifolds (quite generally understood) to Carrollian $\R^\times$-bundles is given in subsection \ref{subsec:CarManRBundles}. The illustrative and physically relevant examples of the event horizons of the Schwarzschild black hole and non-spinning Thakurta spacetime are presented. In Subsection \ref{subsec:CarRBunConn}, we explore the consequences of equipping a Carrollian $\R^\times$-bundle with a principal connection. It is in this subsection that the existence of a non-degenerate metric is shown, and some direct consequences thereof are given. For example, properties of the associated Levi-Civita connection are explored. We then, in Subsection \ref{subsec:NullGeos}, use null geodesics as a probe to explore the geometry of a Carrollian $\R^\times$-bundle equipped with a principal bundle. The results are physically suggestive as `photons' that carry `Carrollian energy' are dynamic. We end with a few concluding remarks in Section \ref{sec:ConRem}.
\medskip 

\noindent \textbf{Conventions:} By manifold, we mean a smooth manifold that is real, finite-dimensional, Hausdorff, and second-countable. We will generally assume that our manifolds are connected, or when necessary, we will work on a connected component. 
%
%%%%%%%%%%%%%%%%%%%%%%%%
%
\section{Preliminaries of Principal $\R^\times$-bundles}\label{sec:Prelims}
\subsection{Principal $\R^\times$-bundles} In this subsection, we will draw heavily on Bruce, Grabowska \& Grabowski \cite{Bruce:2017} and Grabowski \cite{Grabowski:2013} for details of line bundles and their reformulation in terms of principal $\R^\times$-bundles.  To keep this paper relatively self-contained, we recall the concepts needed later in this paper. Consider a line bundle $\pi : L \rightarrow M$. We can construct a new fibre bundle over $M$ by removing the zero section
$$L \longmapsto L^\times := L \setminus \{0_M\}\,.$$
Note that this new fibre bundle is not a line bundle, but now a principal bundle. We will describe this structure using local coordinates $(x^a, t)$.  We can multiply the fibre coordinate $t$ by any non-zero number. Thus, we have a fibrewise action of $\R^\times :=  \R \setminus \{0\} = \mathrm{GL}(1,\R)$ on $L^\times$. The reason why the non-compact group $\R^\times$ appears in the theory is that real line bundles over a manifold $M$ are classified by a $\Z_2$-valued cohomology of $M$, and the transition functions for these bundles take values in $\R^\times$.   We will change notation and set $P := L^\times$, and refer to it as a \emph{principal $\R^\times$-bundle} or just a \emph{$\R^\times$-bundle}.  The right action, we will refer to as a \emph{homogeneity structure}, and in local coordinates it is given as
\begin{align}\label{eqn:RtimesAct}
\rmh  & : P \times \R^\times \longrightarrow P\\
\nonumber & \rmh^*_s (x^a, t) = (x^a, st)\,.
\end{align}
The inverse procedure of passing from a principal $\R^\times$-bundle we denote as $P \mapsto \bar{P} =: L$. The fundamental vector field of the action \eqref{eqn:RtimesAct} is the Euler vector field on $\bar{P}$ restricted to $P$, which locally is $\Delta_P = t \partial_t$. By minor abuse of nomenclature, we will still refer to the Euler vector field. Note that the Euler vector field provides the space of vertical vector fields on $P$ with a global basis. That is, every vertical vector field can globally be written as $X_\sV = f \Delta_P$, with $f \in C^\infty(P)$. \par
To set some terminology,  a tensor or tensor-like object $T$ on $P$ is said to be homogeneous and of weight $r\in \R$, if  
$$\mathcal{L}_{\Delta_P} T = r \, T\,,$$
where $\mathcal{L}_{\Delta_P}$ is the Lie derivative with respect to the Euler vector field. 
\begin{proposition}
The association $L \mapsto P =: L^\times$ establishes a one-to-one correspondence between line bundles over $M$ and principal $\R^\times$-bundles.
\end{proposition}
The above proposition allows us to work with principal $\R^\times$-bundles rather than line bundles when it is convenient to do so. 
\begin{proposition}
Let $\pi : P \rightarrow M$ be a principal $\R^\times$-bundle. Then $\sT P$ is a principal $\R^\times$-bundle with $\R^\times$
$$ \sT \rmh  : \sT P \times \R^\times  \rightarrow \sT P\,.$$
\end{proposition}
We can use adapted coordinates on $\sT P$ $(x^a, t , \dot{\mathbf{t}}, \dot{x}^b)$, where $\dot{\mathbf{t}} :=  t^{-1} \dot{t}$. The $\R^\times$-action is then
$$(\sT\rmh_s)^*(x^a, t , \dot{\mathbf{t}}, \dot{x}^b) = (x^a, st , \dot{\mathbf{t}}, \dot{x}^b)\,.$$
 The admissible coordinate transformations are 
 \begin{align}\label{eqn:AdmCooTrasTP}
 & x^{a'} = x^{a'}(x), && t' = \phi(x)t,\\ \nonumber
 & \dot{\mathbf{t}}' = \dot{\mathbf{t}} + \phi^{-1}(x)\dot{x}^b\frac{\partial \phi}{\partial x^b}(x), && \dot{x}^{b'} =  \dot{x}^a\frac{\partial x^{b'}}{\partial x^a}\,.
 \end{align}
 We have several fibre bundle structures here. We set $(\sT P)_0 :=  \sT P\slash \R^\times$, which is the Atiyah bundle of $P$. The Atiyah bundle comes with adapted coordinates   $(x^a,\dot{\mathbf{t}}, \dot{x}^b)$, and we can identify its sections with $\R^\times$-invariant vector fields on $P$; in other words, vector fields of weight zero.  
\begin{center}
\leavevmode
\begin{xy}
(0,20)*+{\sT P}="a"; (20,20)*+{P}="b";%
(0,0)*+{(\sT P)_0}="c";%
(20,0)*+{M}="d";%
{\ar "a";"b"}?*!/_3mm/{ };%
{\ar "b";"d"}?*!/_4mm/{};{\ar "c";"d"}?*!/^4mm/{};%
{\ar "a";"c"}?*!/_4mm/{};%
\end{xy}
\end{center}
We also have $\sT P \stackrel{\sT \pi}{\longrightarrow} \sT M$ given (symbolically) by $(x^a, t , \dot{\mathbf{t}}, \dot{x}^b) \mapsto (x^a, \dot{x}^b)$. The vertical bundle $\sV P := \ker(\rmd \pi)$, is locally given by $\dot{x}^b =0$, and so comes with adapted coordinates $(x^a, t , \dot{\mathbf{t}})$. Note that the admissible changes of coordinates are 
$$x^{a'} = x^{a'}(x)\,, \qquad t' = \phi(x)t\,, \qquad  \dot{\mathbf{t}}' = \dot{\mathbf{t}}\,.$$
Thus, via this choice of coordinates, we see that $\sV P  \simeq P \times \R$. The vertical bundle being trivial follows from our earlier observation that the vertical vector fields admit a global basis. Thus, we have the identification 
$$\Sec(\sV P)\simeq \Span_{C^\infty(P)}\left \{\Delta_P\right \}\,.$$
Of course, the triviality of the vertical bundle of a principal bundle is well-known.  \par 
Given a vector field $X\in \Vect(P)$, we define the \emph{vertical lift} $i_X$ as the vector field on $\sT P$, locally given by 
\begin{equation}\label{eqn:VertLift}
 X = X^a(x,t)\frac{\partial}{\partial x^a} + X^t(x,t) t\frac{\partial}{\partial t} \longmapsto i_X := X^a(x,t)\frac{\partial}{\partial \dot{x}^a} + X^t(x,t) \frac{\partial}{\partial \dot{\mathbf{t}}} \in \Vect(\sT P)\,.
\end{equation}
The \emph{tangent lift} of a vector field $X \in \Vect(P)$ is the vector field $\mathcal{L}_X \in \Vect(\sT P)$ locally given by  
\begin{equation}\label{eqn:LieDer}
\mathcal{L}_X := X^a\frac{\partial}{\partial x^a} + X^t t\frac{\partial}{\partial t} + \left(\dot{x}^b \frac{\partial X^a}{\partial x^b} + \dot{\mathbf{t}} \,t  \frac{\partial X^a}{\partial t} \right)\frac{\partial}{\partial \dot{x}^a} + \left(\dot{x}^b \frac{\partial X^t}{\partial x^b} + \dot{\mathbf{t}} \left(X^t + t \frac{\partial X^t}{\partial t} \right)\right)\frac{\partial}{\partial \dot{\mathbf{t}}}\,.
\end{equation}
By considering symmetric covariant tensors as functions on $\sT P$ that are locally monomials in ``velocity'', then the tangent lift is precisely the Lie derivative and hence our notation. \par
We will regularly consider the Lie derivative of the Euler vector field, which in adapted coordinates is 
\begin{equation}\label{eqn:LieDerEul}
\mathcal{L}_{\Delta_P} =  t\frac{\partial}{\partial t} + \dot{\mathbf{t}}\frac{\partial}{\partial \dot{\mathbf{t}}}\,.
\end{equation}
 Recall that the Lie algebra of $\R^\times = \mathrm{GL}(1,\R)$ is $\R$ equipped with the standard commutator. The fact that the Lie algebra is just $\R$ vastly simplifies the discussion of connections as compared with more general Lie groups. Importantly, we have a single fundamental vector field, and connection forms are real-valued.  
%
%%%%%%%%%%%%%%%%%%%%%%%%%%%
%
\subsection{Connections on $\R^\times$-bundles}
Our approach to connection is adapted from Kolar et al. \cite[Chapter III]{Kolar:1993}. As we are dealing with one-dimensional vertical bundles, the formalism simplifies. For completeness, we include the relevant details here.
\begin{definition}\label{def:RXConn} An $\R^\times$-connection on a principal $\R^\times$-bundle is a smooth bundle map $\Phi: \sT P \rightarrow \sT P$ over $P$, such that
\begin{enumerate}[i)]
\setlength\itemsep{1em}
\item $\Phi  \circ \Phi = \Phi$,
\item $\Image \Phi =  \sV P$,
\item $(\sT \rmh)_s \circ \Phi = \Phi \circ (\sT \rmh)_s$, for all $s\in \R^\times$.
\end{enumerate}
\end{definition}
As $\ker \Phi$ is of constant rank, it is a vector subbundle of $\sT P$, the \emph{horizontal bundle} and is denoted $\sH P$. By construction, we have
$$\sT P = \sH P\oplus \sV P\;,$$
and $\sT_p P = \sH_p P\oplus \sV_pP$, for all $p \in P$. Thus, given $\Phi$, we recover the classical description of a connection as a specification of a horizontal bundle complementary to the vertical bundle.\par 
We define the connection one-form via
\begin{equation}\label{eqn:ConOneForm}
\Phi(X) = \omega(X)\Delta_P\,,
\end{equation}
for all $X \in \Vect(P)$.  Note that the one-form $\omega$ is unique, and as the Lie algebra of $\R^\times$ is $\R$, we have a genuine one-form on $P$.  From the definition of $\Phi$, we have $\Phi(\Delta_P) = \Delta_P$, and so $\omega(\Delta_P) =  \mathbf{1}$, where $\mathbf{1}$ is the unit function in $C^\infty(P)$. Moreover, the connection one-form must be of weight zero so that $\Phi$ respects the weight. Informally, we view $\omega$ as a one-form `dual' to the Euler vector field. Clearly, $\ker \omega = \sH P$.   \par 
We should carefully check that \eqref{eqn:ConOneForm} does indeed produce a connection one-form.  First, we denote the linear map $\sigma : \R \rightarrow \Vect(P)$, given by $\sigma(r) = r \Delta_P$.  This map sends elements of the Lie algebra of $\R^\times$ to their fundamental vector fields. To be a connection one-form, we need the following to be satisfied
\begin{enumerate}[1.]
\item $\omega(\sigma(r)) = r\cdot \mathbf{1}$, for all $r \in \R$,\label{enu:1}
\item $(\sT \rmh)^*_s \omega = \omega$, for all $s\in \R^\times$.\label{enu:2}
\end{enumerate}
We observe that $\omega(\sigma(r)) =  \omega(r\Delta_P) = r \omega(\Delta_P) = r\cdot \mathbf{1}$, so \ref{enu:1} is satisfied. As $\omega$ is of weight zero, \ref{enu:2} holds.
\begin{definition}
The \emph{connection one-form} of an $\R^\times$-connection $\Phi$ on $P$, is the one-form $\omega$ defined by \eqref{eqn:ConOneForm}.
\end{definition}
In adapted coordinates, the connection one-form can be written as 
$$\omega = \dot{\mathbf{t}} + \dot{x}^a  A_a(x)\,, $$
where $A_a(x)$ is the $\R$-valued gauge field associated with the connection one-form. For the local expression to be invariant under coordinate changes \eqref{eqn:AdmCooTrasTP}, we require 
$$A'_ {a'} = \left(\frac{\partial x^a}{\partial x^{a'}}\right)A_a(x) - \left(\frac{\partial x^a}{\partial x^{a'}}\right) \phi^{-1}\frac{\partial \phi}{\partial x^a}\,.$$
The associated covariant derivative we write as $D_a := \partial_a + A_a(x)\Delta_P$, and this provides a local frame for $\Sec(\sH P)$. In particular, any vector field can be split as $X = X_\sH + X_\sV$, where locally $X_\sH = X^aD_a$ and $X_\sV =  X^t \Delta_P$.\par   
The next proposition is well-known from the general theory of principal connections.
\begin{proposition} \label{prop:RXConnExist}
Let $P$ be a principal $\R^\times$-bundle. Then $\R^\times$-connections on $P$ always exist.
\end{proposition}
\begin{proof}[Sketch of Proof]
Pick a cover $\{U_i\}_{i\in \mathcal I}$ of $M$ that trivialises the principal bundle, i.e., $\pi^{-1}(U_i) \simeq U_i \times \R^\times$, and allows us to employ adapted local coordinates $(x^a_i , t_i)$ on $P$ and $(x^a_i , t_i, \dot{\mathbf{t}}_i, \dot{x}^a_i)$ on  $\sT P$. Over each open $U_i$, define $\Phi_i(X) = \dot{\mathbf{t}}_i(X)\Delta_P$, so  $\omega_i = \dot{\mathbf{t}}_i$. Choose a partition of unity $\sum_i \varphi_i = 1$, subordinate to the chosen cover; such partitions of unity always exist. We then define  $\omega = \sum_i \varphi_i \omega_i$. Note that each $\omega_i$ is a real-valued function and that $\Delta_P$ is globally defined. This simplifies the standard proof, in particular, (\ref{enu:1}. and  \ref{enu:2}.) are satisfied.  
\end{proof}
%
%%%%%%%%%%%%%%%%%%%%%%%%%
%
\section{Carrollian Geometry}\label{sec:CarrGeom}
\subsection{Carrollian $\R^\times$-bundles}\label{subsec:CarRBundles}
Our definition of a Carrollian structure is very similar to that given by Ciambelli et al. \cite{Ciambelli:2019} (a similar bundle approach also appears in \cite{Blitz:2024}), except that we will work with principal $\R^\times$-bundles rather than line bundles, or more generally, fibre bundles with typical fibre $\R$. 
\begin{definition}\label{def:CarrBun}
A principal $\R^\times$-bundle $P$ is said to be a \emph{Carrollian $\R^\times$-bundle} if it is equipped with a degenerate metric on the total space $g$ such that 
$$\ker(g) := \left \{X \in \Vect(P)  ~|~  g(X,-)=0 \right \} =  \Sec(\sV P)\,.$$
\end{definition}
We will denote a Carrollian $\R^\times$-bundle as a pair $(P, g)$, the Euler vector field is canonically derived from the principal $\R^\times$-bundle, and so we do not explicitly state it as part of the structure. Recall that the Euler vector field generates $ \Sec(\sV P)$. \\ 
\noindent  \textbf{Comments:} \
 \begin{enumerate}
 \item The bundle $P$ may be non-trivial, but the vertical bundle $\sV P$ is trivial.
 \item  The degenerate direction of the metric is along the $\R^\times$-fibres. 
 \item  The Euler vector field $\Delta_P$ describes how points $p \in P$ move under the action of $\R^\times$ on $P$. Infinitesimally, we have $(x^a, t \big) \mapsto (x^a , t + \epsilon \, t )$. 
 \item Vertical vector fields can globally be written as $X_\sV = X^t \Delta_P$, where $X^t \in C^\infty(P)$. They correspond to (infinitesimal) gauge transformations. In particular, weight $-1$ vertical vector fields are (locally) of the form $X_\sV = (X^t(x)t^{-1})\Delta_P =  X^t(x)\partial_t$, and thus generate the (infinitesimal) gauge transformations
$$(x^a, t) \mapsto \big(x^a, t + \epsilon\, X^t(x) \big)\,. $$
Such transformations are interpreted as ``supertranslations'' in the BMS context. Note that these (infinitesimal) gauge transformations are abelian, and so the Lie algebra is abelian; this is just the statement that the Lie bracket restricted to weight $-1$ vertical vector fields is trivial. 
 \end{enumerate}
We consider the degenerate metric as a function on $\sT P$ that is locally in adapted coordinates a quadratic monomial in the `velocities'.  We define $g(X,Y) = i_X i_Y g$, for vector fields $X,Y \in \Vect(P)$.
\begin{proposition}
Let $(P, g)$ be a Carrollian $\R^\times$-bundle. Then, in adapted local coordinates, the degenerate metric is of the form
$$g = \dot{x}^a \dot{x}^b g_{ba}(x,t)\,.$$
\end{proposition}
\begin{proof}
It is sufficient to consider $g(\Delta_P,-):= i_ {\Delta_P}g=0$ as by definition $\ker(g)$ is identified with $\Sec(\sV P)$.  In adapted coordinates, an arbitrary degenerate metric on $P$ is of the form
$$g =  \dot{x}^a \dot{x}^b g_{ba}(x,t) + 2 \dot{\mathbf{t}} \dot{x}^ag_a(x,t)+ \dot{\mathbf{t}}^2 g(x,t)\,.$$
As in adapted coordinates $\Delta_P = t \partial_t$, a quick calculations show
$$g(\Delta_P,  - ) = 2 \dot{x}^a g_{a}(x,t) + 2 \dot{\mathbf{t}} g(x,t)\,.$$  
Thus, for this to identically vanish for all values of ``velocity'', $g_a(x,t) =0$ and $g(x,t)=0$. It is observed that the  form of local expression for the degenerate metric does not change under admissible coordinate transformations \eqref{eqn:AdmCooTrasTP}. 
\end{proof}
If $\Delta_P$ is Killing, i.e., $\mathcal{L}_{\Delta_P}g=0$ (see \eqref{eqn:LieDerEul}), then the degenerate metric is weight zero, which means, locally, that the degenerate metric is independent of the fibre coordinate $t$.  This condition is equivalent to saying the degenerate metric is $\R^\times$-invariant (the action here is $(\sT \rmh_s)^* : C^\infty(\sT P) \rightarrow C^\infty(\sT P)$).  If the metric is of weight $k$, i.e., $\mathcal{L}_{\Delta_P} g = k \, g$, then the metric must depend on the fibre coordinate and the degenerate metric scales under $t \mapsto \lambda \,t$ as $g \mapsto \lambda^k\, g$ for any strictly positive $\lambda$. In general, a degenerate metric need not be homogeneous and show scaling behaviour. We will show that a weight-zero degenerate metric induces a non-degenerate metric on $M$.  \par 
We can think of $g$ as a block matrix (in hopefully clear notation)
$$g = \begin{pmatrix}g_M & 0 \\ 0 & 0\end{pmatrix}\,,$$
where $g_M$ is non-degenerate (i.e., as a matrix $\det(g_M)\neq 0$) and viewed as a generalised metric on $M$: locally $g_M$ can, in general, depend on the coordinate $t$, and so is not a genuine metric on $M$.  
\begin{proposition}\label{prop:KillingEulerMetric}
Let $(P,g)$ be a Carrollian $\R^\times$-bundle such that the Euler vector field is Killing, i.e., $\mathcal{L}_{\Delta_P}g =0$. Then the induced metric $g_M$ on $M$ is non-degenerate.
\end{proposition}
\begin{proof}
In adapted coordinates, the weight-zero degenerate metric has the form
$$g = \dot{x}^a \dot{x}^b g_{ba}(x)\,,$$
i.e., no dependence on the coordinate $t$.  Due to the admissible coordinate changes, $g$ is viewed as a basic function of the projection $\sT \pi : \sT P \rightarrow \sT M$. Thus, we naturally view the degenerate metric as a function on $\sT M$, and for distinction we denote this as $g_M$. Consider a vector field $X_M\in \Vect(M)$ and any vector field $X\in \Vect(P)$ that projects to $X$. Locally, we observe that $X = X^a\partial_a + X^t\Delta_P$, with $X_M = X^a\partial_a$. By definition, $\ker(g) = \Sec(\sV P)$, which means that the kernel only contains vertical vector fields, implying that $X$ is not in $\ker(g)$.  Then $g_M(X_m, - ) \neq 0$ (other than $X_M =0$). Thus, $g_M$ is non-degenerate.
\end{proof}
\begin{proposition} \label{prop:Killing}
Let $(P,g)$ be a Carrollian $\R^\times$-bundle such that the Euler vector field is Killing, i.e., $\mathcal{L}_{\Delta_P}g =0$. Then all vector fields in $\Sec(\sV P)$ are Killing.
\end{proposition}
\begin{proof}
As $\sV P$ is trivial, all vertical vector fields can be written as $X_\sV = X^t \Delta_P$, with  $X^t \in C^\infty(P)$. Using \eqref{eqn:LieDer}, we locally have
$$\mathcal{L}_{X_\sV} = X^t \Delta_P + \left( \dot{x}^a \frac{\partial X^t}{\partial x^a} + \dot{\mathbf{t}}(X^t +  \Delta_P X^t)\right)\frac{\partial}{\partial \dot{\mathbf{t}}}\,.$$ 
Then as $g$ has no dependence on the coordinate $\dot{\mathbf{t}}$, we observe that globally $\mathcal{L}_{X_\sV}g = X^t\mathcal{L}_{\Delta_P}g$. Thus, if $\mathcal{L}_{\Delta_P}g =0$, then $\mathcal{L}_{X_\sV}g=0$.
\end{proof}
In other words, if the Euler vector field is Killing, then $\ker(g) = \Sec(\sV P)$ consists of Killing vector fields. This implies that the Lie algebra of vertical Killing vector fields is infinite-dimensional. We view these vertical vector fields as generating trivial symmetries; they are due to the degeneracy rather than genuine symmetries.  The result that rescaling a Killing vector field by any (non-constant) function is again a Killing vector field is not true in standard (pseudo-)Riemannian geometry - the Lie algebras of Killing vector fields are finite. We will further examine Killing vector fields in the next subsection by using a connection to separate the horizontal and vertical vector fields (see Proposition \ref{prop:KillAreProj}). 
\begin{example}
Consider the trivial $\R^\times$-bundle $P = \R^n\times \R^\times$, which we equip with global coordinates $(x^a, t)$. We can give $P$ the structure of a Carrollian $\R^\times$-bundle using the standard Euclidean metric on $\R^n$, i.e., $g = \dot{x}^a \dot{x}^b\delta_{ba}$ is a degenerate metric on $\R^n\times \R^\times$. The Euler vector field is (globally) $\Delta_P = t \partial_t$, and clearly $\mathcal{L}_{\Delta_P}g=0$ as $g$ has no dependence on $t$.
\end{example}
\begin{example}\label{exp:RiemMan}
More general than the previous example, consider a (pseudo)-Riemannian $(M,g_M)$. We then consider $\pi: M \times \R^\times\rightarrow M$ equipped with the pullback metric $g := (\sT \pi^*)g_M$.  Clearly,  $\mathcal{L}_{\Delta_P}g=0$ as $g$ does not depend on $t$.  As a specific example, we have $P = S^1 \times \R^\times $, where we have equipped $S^1$ with the circle metric $g_{S^1}$.
\end{example}
\begin{example}
We define an equivalence relation on $S^1\times \R^\times$ as $(p, r) \sim (p +2 \pi, -r)$. The M\"obius $\R^\times$-bundle $\pi : P\rightarrow S^1$ is defined as $P := (S^1\times \R^\times)\slash \sim$. The projection is the projection onto the first factor. We can equip $\tau : S^1\times \R^\times\rightarrow S^1$ with a degenerate metric $g$, by defining $g :=  (\sT \tau)^*g_{S^1}$, where we have the circle metric $g_{S^1}$ on $S^1$. This degenerate metric induces a degenerate metric on $P$ using the equivalence relation; as the degenerate metric on $S^1\times \R^\times$ does not depend on points in $\R^\times$, the metric on $P$ is well-defined. Similarly, it is clear that $g(\Delta_P,-) =0$, and so we have a Carrollian $\R^\times$-bundle with a non-trivial bundle structure.
\end{example}
\begin{example}\label{exm:RiemGeo}
Consider any principal $\R^\times$-bundle $\pi : P \rightarrow M$ over a Riemannian manifold $(M, g_M)$.  Then $g := (\sT \pi)^* g_M$ is a degenerate metric on $P$. To check that we have a Carrollian structure, notice that $g$ has no components in the fibre direction. As $g_M$ is non-degenerate, $g(X, -) = 0$ if and only if $X \in \Sec(\sV P)$. Thus, we have a Carrollian $\R^\times$-bundle. Moreover, $\mathcal{L}_{\Delta_P}g = 0$ as $g$ is independent of the fibre direction.
\end{example}
\begin{definition}
Let $(P,g)$ and $(P',g')$ be Carrollian $\R^\times$-bundles. An \emph{isomorphism of Carrollian $\R^\times$-bundles} is a bundle isomorphism $\phi : P \rightarrow P'$ (over $\tilde{\phi}$), i.e., a diffeomorphism that satisfies $\pi' \circ \phi =  \tilde{\phi} \circ \pi$, such that 
\begin{enumerate}[1.]
\item it is $\R^\times$-equivariant, i.e., $\rmh_s(\phi(p)) = \phi(\rmh_s(p))$, for all $s\in \R^\times$ and $p\in P $, and
\item  it is an isometry, i.e., $(\sT\phi)^*g' = g$.
\end{enumerate}
\end{definition}
The $\R^\times$-equivariant condition can also be expressed as 
\begin{equation}
\Delta_P \circ \phi^* =  \phi^* \circ \Delta_{P'}\,.
\end{equation} 
%
%%%%%%%%%%%%%%%%%%%%%%%%%%%
%
\subsection{Carrollian Manifolds as Carrollian $\R ^\times$-bundles}\label{subsec:CarManRBundles}
We will take a quite general definition of a Carrollian manifold. First, we recall the standard definition of a weak Carrollian manifold, due to Duval et al. (see \cite{Duval:2014}). Similar structures have been referred to as null manifolds by Mars \cite{Mars:2024}. For material on foliations, the reader may consult \cite[Chapter 1]{Moerdijk:2003}.
\begin{definition}[Duval et al. \cite{Duval:2014}]
A \emph{weak Carrollian manifold} is a triple $(N, h, \zx)$, where $N$ is a smooth manifold, $h$ is a degenerate metric, $\zx$ is a nowhere vanishing complete vector field, called the \emph{Carroll vector field}, such that $\ker(h) = \Span\{\zx\}$.
\end{definition}
Note that, as it stands, while the above definition ensures that the associated foliation is regular, i.e., all the leaves are one-dimensional, it does not imply a simple foliation. We will change focus from the Carroll vector field to the distribution it defines and allow more general behaviour that captures the possibility of a Carroll vector field being singular and not complete.
\begin{definition}\label{def:CarrMan}
A \emph{Carrollian manifold} is a pair $(N, h)$, where $N$ is a smooth manifold equipped with a degenerate metric $h$, whose kernel is a possibly singular distribution with leaves of maximal dimension one. If the associated foliation of $N$ is regular, simple and with non-compact leaves, then we will refer to a \emph{simple Carrollian manifold}.
\end{definition}
A simple foliation is a foliation by leaves of a surjective submersion. For the case at hand, a simple Carrollian manifold is foliated by leaves diffeomorphic to $\R$. The coordinates on the leaves are interpreted as time, implying that simple Carrollian manifolds must be non-compact manifolds. We point out that simple Carrollian manifolds are, once a basis for the distribution has been chosen,  examples of weak Carrollian manifolds where the Carroll vector field is non-singular and complete, and the associated foliation is simple and has non-compact leaves.  Typical examples in the literature have a trivial bundle structure, i.e., $N = M \times \R$. However, we allow for non-trivial bundle structures. 
\begin{remark}
The definition we give of a Carrollian manifold is very general in that, by itself, there is no reason why the associated foliation must be regular and simple, there could be leaves of varying topological types and dimension e.g., mixtures of leaves diffeomorphic to $\R$ and $S^1$, and possibly single points where the dimension drops to zero. Even if the foliation is simple, the leaves could be diffeomorphic to $S^1$; we would then be discussing circle bundles. However, as Carrollian manifolds in physics arise as the ultra-relativistic limit of a spacetime, understood as globally hyperbolic and time-orientable Lorentzian manifolds, they possess a non-compact time dimension. Thus, a large class of physically interesting Carrollian manifolds will be simple Carrollian manifolds. 
\end{remark}
\begin{coexample}
Consider $\R^2\setminus \{(0,0)\}$ which we equip with standard coordinates $(x,y)$, and the degenerate metric $h = (\dot{x}\, x + \dot{y}\, y)$. Via inspection, one can quickly see that the non-vanishing vector field  $\zx = -y \partial_x + x \partial_y$ generates the rank-one kernel. However, the integral curves are concentric circles about the removed point $(0,0)$, and so this is not an example of a simple Carrollian manifold.  Similarly, considering $\R^2$, the vector field $\zx$ vanishes, as does the degenerate metric, at $(0,0)$, and the kernel becomes rank-two. 
\end{coexample}
Given a simple  Carrollian manifold $(N,h)$, we denote the line bundle defined by the kernel of the degenerate metric as $\tau : L \rightarrow N$. Thus, the tangent bundle splits as 
$$\sT N \simeq L\oplus \sT N\setminus L =: L \oplus E\,,$$
where the fibres of $E$ are $E_n = \sT_n N\setminus L_n$, for all points $n \in N$. Clearly, $\mathsf{rank}(E) = \mathsf{dim} N -1$.   For a simple Carrollian manifold, there are codimension-one embedded submanifolds of $N$ that are transversal to the foliation, i.e., the tangent bundle of the submanifold does not contain elements from the line bundle defining the foliation. We will refer to such submanifolds as \emph{transversal slices}. An important property of transversal slices is that they intersect each leaf at a single point. For non-simple Carrollian manifolds, such submanifolds can only be locally defined. 
\begin{example}
Consider $\R^2$ with the simple foliation with leaves given by $\mathsf{L}_c = \{(x,y)\in \R^2 ~~~|~~ y = c\}$, for any constant $c =\R$. Clearly, the leaves are all diffeomorphic to $\R$. An example of a transversal slice is the $y$-axis, i.e., $ \mathsf{S} =\{(x,y)\in \R^2 ~~|~~ x = 0\}\subset \R^2$.
\end{example}
\begin{example}\label{exp:TransSlice}
Let $N = M \times \R$, where $M$ is any smooth manifold. Consider the simple foliation of $N$ with leaves given by $\mathsf{L}_m = \{m\}\times \R$, where $m \in M$. Clearly, the leaves are all diffeomorphic to $\R$. An example of a transversal slice is any copy of $M$ in $N$, e.g., $\mathsf{S} = M \times \{0\}$. However, other valid choices of a transversal slice are $\mathsf{S}_r := M \times \{r\}$, for any $r \in \R$. Even more generally, any transversal slice can be defined as the graph of a smooth function $f : M \rightarrow \R$ and setting $\mathsf{S}_f := \{\big(m, f(m)\big) \in M \times \R ~~|~~ m \in M\}$.
\end{example}
\begin{theorem}\label{thrm:CarrManAreCarrBun}
For any simple Carrollian manifold $(N,h)$, there exists a Carrollian $\R^\times$-bundle $(P, g)$ that can be non-canonically constructed from $(N,h)$.
\end{theorem}
Before we can prove the above theorem, we will need a lemma. Note that we will be dealing with fibre bundles with typical fibre $\R$, but not necessarily those whose transition maps are linear. We want a (possibly non-canonical) way of associating a line bundle with a more general bundle with typical fibre $\R$. The approach is to use a Taylor expansion of the transition maps (we will shortly be more precise), but in order to do this, we need to select a section of $E$ in order to make sense of setting the fibre coordinates to zero in a consistent way.  The choice of section introduces the non-canonical aspect of finding such line bundles. If we have a line bundle from the start, then we may simply choose the zero section, and the following lemma is trivial. 
\begin{lemma}\label{lem:LinFibBun}
Let $\pi: E \rightarrow M$  be a fibre bundle with typical fibre $\R$ over a paracompact and connected smooth manifold $M$. Then, for any chosen global section $s \in \Sec(E)$, there exists a unique (up to isomorphism) line bundle $L_s(E)\rightarrow M$, and a fibre-preserving diffeomorphism $\Phi_s : L_s(E)\stackrel{\sim}{\rightarrow} E$.
\end{lemma}
\begin{proof}
As $\R$ is contractible and we have taken $M$ to be paracompact and connected, $E$ admits global sections (using a partition of unity, we can extend local sections to global ones). Given local trivialisations 
$$\phi_i : \pi^{-1}(U_i) \stackrel{\sim}{\rightarrow} U_i \times \R\,, \qquad \phi_i(p)= (m, r)\,, $$
the associated transition functions we write as
$$\phi_{ij}(m, r) := (m, \psi_{ij}(m,r))\,.$$
Let us fix a global section $s\in \Sec(E)$. We define a shifted local trivialisation associated with $s$ by defining  
$$\tilde{\phi}_i(p) = (m, r- s_i(m)) := (m, r_i)\,.$$
where $s_i = s|_{U_i}$. We have shifted the global coordinate on $\R$ over $U_i$ using the section.  We set $r_i:= r -s_i(m)$, where $r$ is the original coordinate. The shifted transition functions $\tilde{\phi}_{ij} := \tilde{\phi}_i \circ \tilde{\phi}^{-1}_j$ are then
$$\tilde{\phi}_{ij}(m, r_j) = (m, \psi_{ij}(m, r_j + s_j(m))- s_i(m)) =: (m, \tilde{\psi}_{ij}(m,r_j)) \,.$$
Setting $r_j =0$, we observe that 
$$\tilde{\phi}_{ij}(m, 0) = (m, \psi_{ij}(m, s_j(m))- s_i(m)) = (m, 0) \,,$$
as $s_i = s|_{U_i}$ and $s_j = s|_{U_j}$ implies that $s_i(m) = \psi_{ij}(m,s_j(m))$. Thus, by construction, we have $\tilde{\phi}_{ij}(m, 0) = (m, 0)$. This is not, in general, the case for the original transition functions $\phi_{ij}$ as they are assumed to be non-linear. \par 
Taylor expanding $\tilde{\psi}_{ij}(m, r_j)$ about $r_j =0$, we obtain $\tilde{\psi}_{ij}(m, r_j) = 0 + \big( \tilde{\psi}'_{ij}(m, 0)\big)r_j + \cdots$. Importantly, the first-order term in the expansion satisfies the cocycle condition independently of higher-order terms. Thus, given a local trivialisation of $E$ and a chosen (global) section, we construct a line bundle over $M$ by defining the transition functions as 
$$\varphi_{ij}(m,t) := \left(m, \big( \tilde{\psi}'_{ij}(m, 0)\big)t \right)\,.$$
We denote this line bundle as $L_s(E)$, and note that, via a standard result in bundle theory, this line bundle is unique up to isomorphism. \par 
The fibre-preserving diffeomorphism $\Phi_s : L_s(E) \stackrel{\sim}{\rightarrow} E$, is defined in a trivialisation by $\Phi^*(x^a, r_i) = (x^a, t_i + s_i(x))$, where $s_i(x)$ is the coordinate expression for $s$.  
\end{proof}
\begin{proof}[Proof of Theorem~ \ref{thrm:CarrManAreCarrBun}]
Let $N$ be of dimension $n+1$. A simple Carrollian manifold $(N,h)$ is endowed with a rank-one distribution whose integral submanifolds (leaves) form a simple foliation $\mathcal{F}$. As we have a simple foliation with the leaves being diffeomorphic to $\R$, we know that $N$ is the total space of a fibre bundle $\pi:N \rightarrow M$, with typical fibre $\R$, and the base manifold is the leaf space $M:=N / \mathcal{F}$, which is an $n$-dimensional manifold. Moreover, $M$ is paracompact and connected. \par 
We want to construct an atlas for this fibre bundle. A key step is identifying the base manifold $M$ as a submanifold of $N$. There is no canonical choice here. However, as we have a simple foliation, we know that global transversal slices exist, and these serve as a choice of embedding the base manifold inside the total space. Let us choose a transversal slice $\bar{M}\subset N$. We identify this slice with the base manifold $M$ via an embedding $\iota:M \rightarrow N$ such that $\iota(M) = \bar{M}$. With this choice made, we have local trivialisations 
$$\phi_i : \pi^{-1}(U_i) \stackrel{\sim}{\rightarrow} U_i \times \R\,, \qquad \phi_i(p)= (m, r)\,, $$
and associated transition functions 
$$\phi_{ij}(m, r) := (m, \psi_{ij}(m,r))\,.$$
This bundle atlas need not be fibrewise linear. Thus, in general, we do not obtain a line bundle from a choice of transversal slice. From Lemma \ref{lem:LinFibBun}, we can linearise $\pi : E \rightarrow M$ by selecting a global section (which is guaranteed to exist). We denote the linearisation (about a section $s\in \Sec(N)$)  as $L_s(N,\bar{M})$ and the associated fibre-preserving diffeomorphism as $\Phi_s : L_s(N,\bar{M}) \rightarrow N$. Thus, given any simple Carrollian manifold, we can non-canonically associate with it a line bundle via a choice of transversal slice and a section.  The associated $\R^\times$-principal bundle we denote as $P_s(N, \bar{M}):= L_s(N, \bar{M})^\times$.\par
The degenerate metric $h$ on $N$, induces a degenerate metric on $L_s(N, \bar{M})$, i.e., we define $\bar{h} :=  (\sT \Phi_s)^* h$. We have a smooth inclusion (just delete the zero in each fibre) $ \i : P_s(N, \bar{M}) \rightarrow L_s(N, \bar{M}) $, and so $\sT \i : \sT P_s(N, \bar{M}) \rightarrow \sT L_s(N,\bar{M})$ defines a submanifold. The degenerate metric on $P_s(N, \bar{M})$ is then defined as $g = (\sT \i)^*\bar{h}$. \par 
 We need to check that $\ker(g) = \Sec(\sV P(N))$. We have a smooth inclusion $\Phi_s \circ \i : P_s(N, \bar{M})\rightarrow N$, which induces smooth inclusions
\begin{align*}
\sT(\Phi_s \circ \i)|_{\sV P_s(N, \bar{M})}: &\, \sV P_s(N, \bar{M}) \rightarrow  \sV N\,,\\
\sT(\Phi_s \circ \i)|_{\ker(g)}: &\, \ker(g) \rightarrow  \sV N\,,
\end{align*}
where we understand $\ker(g)$ as a line bundle in $\sT P_s(N, \bar{M})$, and similarly, we make the identification $\ker(h) = \sV N$.\par
Let $v \in \sV P_s(N, \bar{M})$, then $\sT(\Phi_s \circ \i)|_{\sV P_s(N, \bar{M})}(v) = \sT(\Phi_s \circ \i)(v) \in \sV N $. As $\sT(\Phi_s \circ \i)|_{\ker(g)}$ maps $\ker(g)$ to $\sV N$ if $\sT(\Phi_s \circ \i)(v) \in \sV N $, then $v \in \ker(g)$. Thus, $\sV P_s(N, \bar{M}) \subseteq \ker(g)$.  \par 
\smallskip
Similarly, let $v \in \ker(g)$, then $\sT(\Phi_s \circ \i)|_{\ker(g)}(v) = \sT(\Phi_s \circ \i)(v) \in \sV N $. As $\sT(\Phi_s \circ \i)|_{\sV P_s(N, \bar{M})}$ maps $\sV P_s(N, \bar{M})$ to $\sV N$ if $\sT(\Phi_s \circ \i)(v) \in \sV N $, then $v \in \sV P_s(N, \bar{M})$. Thus, $\ker(g)\subseteq\sV P(N)$.\par 
Together, we conclude that (as line bundles) $\ker(g) = \sV P_s(N, \bar{M})$, so as vector fields $\ker(g) = \Sec(\sV P_s(N,\bar{M}))$.
\end{proof}
Theorem \ref{thrm:CarrManAreCarrBun} can diagrammatically be paraphrased as 
$$(N, h) \xrightarrow{\substack{\textnormal{Transversal Slice} \\ \textnormal{+ Linearise}}} (L, \bar{h}) \xrightarrow{\substack{\textnormal{Remove} \\ \textnormal{Zero Section}}} (P,g)$$
\begin{remark}
The construction of a  Carrollian $\R^\times$-bundle is not canonical, as it requires a choice of embedded codimension-one submanifold.  However, in applications, there may be natural submanifolds to consider. The reader is reminded of Example \ref{exp:TransSlice}.
\end{remark}
\begin{example}[Schwarzschild black hole]\label{exa:BHHorizon}
Consider the Schwarzschild black hole $(\mathcal{M} = \R^2 \times S^2, \rmd s^2)$, where we will employ Eddington--Finkelstein coordinates to write the spacetime interval as
$$\rmd s^2 = - \left(1 - \frac{2 G M}{r} \right)\rmd v^2 + 2\, \rmd v \rmd r + r^2 \rmd \Omega^2\,,$$
where $\rmd \Omega^2$ is the interval (metric) on the round sphere $S^2$. Here $M$ is the mass of the black hole and $G$ is Newton's constant.  The event horizon is defined as $N := \{ p\in \mathcal{M} ~~|~~ r(p) = 2GM\} = S^2 \times \R$. The induced degenerate metric on $N$ is $h = (2GM)^2 g_{S^2}$, where $g_{S^2}$ is the round metric on $S^2$. Clearly, $\ker(h) = \Span(\partial_v)$, and thus the event horizon of a Schwarzschild black hole is a simple Carrollian manifold. Note that $\partial_v$ is a Killing vector field.  The points of $N$ can be parametrised as pairs $(\mathbf{u}, v)$, where $\mathbf{u}\in S^2$ and $v \in \R$. We fix a transversal slice $\iota : S^2 \rightarrow S^2 \times \R$ defined by $\mathbf{u} \mapsto (\mathbf{u}, 0)$, i.e., we pick the sphere at $v =0$.  The  bundle structure is  $N = S^2 \times \R$ with the projection $(\mathbf{u}, v) \mapsto (\mathbf{u}, 0) = \mathbf{u}$. As we have a trivial bundle with typical fibres $\R$, the linearisation process is trivial; we can just use the zero section. Thus, we identify $L_0(N, S^2)= N$. The principal  $\R^\times$-bundle is given by removing the zero section, and as the line bundle is trivial, $P_0(N, S^2) = S^2 \times \R^\times$, we denote the global $\R$ coordinate as $t$. The projection is $(\mathbf{u}, t) = \mathbf{u}$. The degenerate metric is easily seen to be $g = (\sT \i)^*h = (2GM)^2 g_{S^2}$. The Euler vector field $\Delta =  t \partial_t$ spans $\ker(g)$, which is a section of the vertical bundle. Thus, we have a Carrollian $\R^\times$-bundle. Moreover, as $\mathcal{L}_\Delta g =0$, the Euler vector field $\Delta$ is Killing and so all sections of the vertical bundle are Killing, see Proposition \ref{prop:Killing}.
\end{example}
\begin{example}[Thakurta spacetime]\label{exa:ThakutraSP}
Consider the non-spinning Thakurta spacetime $(\mathcal{M} = \R^2 \times S^2, \rmd s^2)$, where we will employ Eddington--Finkelstein coordinates to write the spacetime interval as
$$\rmd s^2 = \rme^{-U(v)} \left(- \left(1 - \frac{2 G M}{r} \right)\rmd v^2 + 2\, \rmd v \rmd r + r^2 \rmd \Omega^2\right)\,,$$
where $\rmd \Omega^2$ is the interval (metric) on the round sphere $S^2$.  The event horizon is defined as $N := \{ p\in \mathcal{M} ~~|~~ r(p) = 2GM\} = S^2 \times \R$. The induced degenerate metric on $N$ is $h = (2GM)^2 \rme^{-U(v)} g_{S^2}$, where $g_{S^2}$  is the round metric on $S^2$. Clearly, $\ker(h) = \Span(\partial_v)$, and thus the event horizon is a simple Carrollian manifold.  The points of $N$ can be parametrised as pairs $(\mathbf{u}, r)$, where $\mathbf{u}\in S^2$ and $r \in \R$. We fix a transversal slice $\iota : S^2 \rightarrow S^2 \times \R$ defined by $\mathbf{u} \mapsto (\mathbf{u}, 0)$, i.e., we pick the sphere at $r =0$.  The  bundle structure is  $N = S^2 \times \R$ with the projection $(\mathbf{u}, v) \mapsto (\mathbf{u}, 0) = \mathbf{u}$. As we have a trivial bundle with typical fibres $\R$, the linearisation process is trivial; we can just use the zero section. Thus, we identify $L_0(N, S^2)= N$.  The principal  $\R^\times$-bundle is given by removing the zero section, and as the line bundle is trivial, $P_0(N, S^2) = S^2 \times \R^\times$. The projection is $(\mathbf{u}, t) = \mathbf{u}$. The degenerate metric is easily seen to be $g = (\sT \i)^*h = (2GM)^2 \rme^{-U(t)} g_{S^2}$. The Euler vector field $\Delta =  t \partial_t$ spans $\ker(g)$, which is a section of the vertical bundle. Thus, we have a Carrollian $\R^\times$-bundle. Note, $\mathcal{L}_\Delta g  = (1- \dot{U}(t))g \neq 0$, so $\Delta$ is not Killing, but rather a conformal Killing vector field.
\end{example}
%
%%%%%%%%%%%%%%%%%%%%%%%%%%%
%
\subsection{Carrollian $\R^\times$-bundles with a Connection}\label{subsec:CarRBunConn}
Recall that an $\R^\times$-connection on a principal $\R^\times$-bundle is a projection $\Phi : \sT P \rightarrow \sT P$ that allows for the decomposition $\sT P \simeq \sH P \oplus \sV P$, with $ \Image \Phi =  \sV P$ and $\ker \Phi = \sH P$ (see Definition \ref{def:RXConn}). 
\begin{definition}
A \emph{Carrollian $\R^\times$-bundle with a connection} is a triple $(P, g, \Phi)$, where $(P,g)$ is a Carrollian $\R^\times$-bundle and $\Phi$ is a $\R^\times$-connection on $P$.
\end{definition}
Proposition \ref{prop:RXConnExist} tells us that we can always equip a Carrollian $\R^\times$-bundle with an $\R^\times$-connection. This connection is not, in general, canonical and certainly not unique. Given a connection, we have the decomposition
$$\Vect(P) \simeq \Sec(\sH P)\oplus \Sec(\sV P) = \Sec(\sH P)\oplus \ker(g)\,.$$
Writing $X = X_\sH + X_\sV \in \Vect(P) $, we observe that $g(X_\sH, Y_\sV) = 0$ as $Y_\sV \in \ker(g)$. Thus, we have the following direct result.
\begin{proposition}
Let  $(P, g, \Phi)$ be a Carrollian $\R^\times$-bundle with a connection. Then  $\Sec(\sH P)$ and $\ker(g)$ are orthogonal.
\end{proposition}
An $\R^\times$-connection gives a splitting of the tangent space  $\sT_p P \simeq \sH_p P \oplus \sV_p P$ for all points $p  \in P$. We interpret this splitting as locally defining temporal and spatial directions that are orthogonal.
\begin{proposition} Let $(P, g, \Phi)$ be a Carrollian $\R^\times$-bundle with a connection. Then $g^{(+)}_\omega := g + \omega^2$ and   $g^{(-)}_\omega := g - \omega^2$ are non-degenerate metrics on $P$. 
\end{proposition}
\begin{proof}
We will write $g^{(\pm)}_\omega = g \pm \omega^2$. Splitting $X = X_\sH + X_\sV$ is clear that $g^{(\pm)}_\omega(X_\sH,-) = g(X_\sH,-) \neq 0$ and  $g^{(\pm)}_\omega(X_\sV,-) = \pm 2 \,\omega(X_\sV)\,\omega(-) \neq 0$.
\end{proof}
We will consider $g^{(+)}_\omega = g +\omega^2 $ for concreteness, and write this (locally) as a matrix (in hopefully clear notation),
$$g^{(+)}_\omega = \begin{pmatrix}g_M +AA^{\sT} & t^{-1}A \\ t^{-1}A^{\sT} & t^{-2}\end{pmatrix}\,, \qquad g = \begin{pmatrix}g_M & 0 \\ 0 & 0\end{pmatrix}\,.$$
Thus, we have a kind of Kaluza--Klein geometry in which the extra dimension is not compact and the scalar field is the unit function. The Christoffel symbols for the metric $g^{(+)}_\omega$ are related to the  Christoffel symbols of the metric $g_M$ (assuming $\mathcal{L}_{\Delta_P} g =0$) in the same way as they are in standard Kaluza--Klein geometry (see, for example \cite{Kerner:2001}). In essence, there is a Kaluza--Klein theory behind $\R$-connections on Carrollian $\R^\times$-bundles.  We will not pursue this idea further here.
Written out locally, we have 
\begin{equation}\label{eqn:MetOnP}
g^{(+)}_\omega = \dot{x}^a \dot{x}^b( g_{ba}(x,t) + A_aA_b(x)) + 2 \dot{\mathbf{t}}\dot{x}^aA_a(x) + \dot{\mathbf{t}}^2\,.
\end{equation}
\begin{remark}
The term $\dot{\mathbf{t}}^2$ of \eqref{eqn:MetOnP} is interpreted as the fibre metric, or more carefully (as a $1\times 1$ matrix) $g_{tt}(x,t) = t^{-2}$. Thus, once we have a fibre metric $g_{tt}(x,t)$ ($\R$-invariant or not), we can define a new non-degenerate metric on $P$ as   $g'^{(+)}_{\omega} = \dot{x}^a \dot{x}^b( g_{ba}(x,t) + A_aA_b(x)) + 2 \dot{\mathbf{t}}\dot{x}^aA_a(x) + \dot{\mathbf{t}}^2g_{tt}(x,t)$.
\end{remark}
\begin{proposition}\label{prop:TorsionlessCon}
Let  $(P, g, \Phi)$ be a Carrollian $\R^\times$-bundle with a connection. Then $P$ can canonically be equipped with an affine connection that is torsionless, but not, in general, metric compatible. 
\end{proposition}
\begin{proof}
The Levi-Civita connection $\nabla$ of $g^{(\pm)}_\omega$ provides the torsionless affine connection. While the Levi-Civita connection is metric compatible with respect to $g_\omega$, it need not be with respect to $g$. Observe that $\nabla_C (g_\omega)_{ab} =\nabla_C g_{ab}  + \nabla_C A_a A_b + A_a\nabla_C A_b =0 $, where the index $C = (c,t)$. Thus, we do not, in general, have metric compatibility with $g$.
\end{proof}
Coupled with Proposition \ref{prop:RXConnExist}, we observe that given any Carrollian $\R^\times$-bundle $(P,g)$, one can always equip $P$ with a torsionless affine connection. For completeness, the (non-vanishing) Christoffel symbols  for $g^{(+)}_\omega$ are  
\begin{align*}
 & \Gamma^c_{ab} =\tilde{\Gamma}^c_{ab} + \frac{1}{2} (g_M)^{cd} (A_b F_{ad} + A_a F_{bd}) + \frac{1}{2} (g_M)^{cd}A_d t\frac{\partial (g_M)_{ab}}{\partial t}\,,\\
 &\Gamma^c_{at} = \Gamma^c_{ta} = \frac{1}{2} (g_M)^{cd} \left( \frac{\partial (g_M)_{ad}}{\partial t} + \frac{1}{t}F_{da} \right)\,,\\
 &\Gamma^t_{ab} = -t A_c \tilde{\Gamma}^c_{ab} - \frac{t}{2} (g_M)^{cd}A_d (A_b F_{ac} + A_a F_{bc}) + \frac{t}{2}(\partial_a A_b + \partial_b A_a) - \frac{t^2}{2}(1 + (g_M)^{cd}A_c A_d)\frac{\partial (g_M)_{ab}}{\partial t}\,,\\
 &\Gamma^t_{at} = \Gamma^t_{ta} =-\frac{1}{2}(g_M)^{cd}A_d F_{ca} - \frac{t}{2}(g_M)^{cd}A_d \frac{\partial g_{cd}}{\partial t} \,,\\
 & \Gamma^t_{tt}=-\frac{1}{t}\,,
\end{align*} 
where $\tilde{\Gamma}^c_{ab}$  are the Christoffel symbols of the (possibly time dependent) metric $g_M$, and $F$ is the curvature of the principal connection, i.e., $F_{ab} =  \partial_a A_b -\partial_b A_a$.\par 
Notice that two problematic  Christoffel symbols require consideration when using the smooth inclusion $P \hookrightarrow L$, i.e.,  $\Gamma^c_{at} \sim t^{-1}$ and $ \Gamma^t_{tt}\sim t^{-1}$. We will assume that $g_M$ is smooth at $t=0$.  One can quickly observe that for (non-zero) vector fields $X \in \Vect(L)$, the regularity condition that their components are at least linear in $t$ near the zero section must be satisfied. We refer to such vector fields on $L$ as \emph{regular vector fields}. We thus have the following proposition.
\begin{proposition}\label{prop:ExtLCcon}
Let  $(P, g, \Phi)$ be a Carrollian $\R^\times$-bundle with a connection. Then the Levi-Civita connection $\nabla$ associated with $g^{(+)}_\omega$ provides an affine connection on regular vector fields on the line bundle $L$, where $L^\times = P$.
\end{proposition}
\begin{remark}
Proposition \ref{prop:ExtLCcon} directly generalises to the Levi-Civita connection associated with $g^{(-)}_\omega$ upon taking care of the extra minus sign.
\end{remark}
\begin{example}[Schwarzschild black hole]
Continuing Example \ref{exa:BHHorizon}, consider $P = S^2 \times \R^\times$, equipped with the degenerate metric $g = (2GM)^2\, g_{S^2}$. Note that $\mathcal{L}_\Delta g =0$, i.e., the degenerate metric is independent of the fibre coordinate. As the bundle structure is trivial, we can canonically select the trivial connection $\omega = \mathbf{\dot{t}} = \dot{t}\,t^{-1}$. The Riemannian metric on $P$ is $g^{(+)}_\omega = (2GM)^2\, g_{S^2} + \mathbf{\dot{t}}^2$. As the trivial connection is flat, i.e., $F_{ab} =0$ as $A_a =0$, the non-vanishing Christoffel symbols are
$$\Gamma_{ab}^c = \tilde{\Gamma}_{ab}^c\,, \qquad \Gamma_{tt}^t = - t^{-1}\,, $$
where $\tilde{\Gamma}_{ab}^c$ are the Christoffel symbols on $S^2$ equipped with the standard round metric.
\end{example}
\begin{example}[Thakurta spacetime]
Continuing Example \ref{exa:ThakutraSP}, consider $P = S^2 \times \R^\times$, equipped with the degenerate metric $g = (2GM)^2\, \rme^{-U(t)} g_{S^2}$. As the bundle structure is trivial, we can canonically select the trivial connection $\omega = \mathbf{\dot{t}} = \dot{t}\,t^{-1}$. The Riemannian metric on $P$ is $g^{+}_\omega = (2GM)^2\,\rme^{-U(t)} g_{S^2} + \mathbf{\dot{t}}^2$. As the trivial connection is flat, i.e., $F_{ab} =0$ as $A_a =0$, the non-vanishing Christoffel symbols are
\begin{align*}
& \Gamma_{ab}^c = \tilde{\Gamma}_{ab}^c\,, && \Gamma_{at}^c = \Gamma_{ta}^c =  -\frac{1}{2}\dot{U}(t) \, \delta_a^c\,, \\
& \Gamma_{ab}^t =  \frac{t^2}{2} \dot{U}(t)\, g_{ab}\,,&& \Gamma_{tt}^t = - t^{-1}\,,
\end{align*}
\end{example}
As $g^{(+)}_\omega$ is non-degenerate, we have $\det(g^{(+)}_\omega) =  t^{-2}\det(g_M) \neq 0$ (though $\det(g)=0$). Thus, the canonical volume is (locally)
$$\mathsf{Vol}_P = t^{-2}\det(g_M) =   \sqrt{|g_M|}(x,t) t^{-1}\, \rmd x^1\wedge \cdots \wedge \rmd x^n\wedge \rmd t\,.$$
The divergence operator is then defined as standard, $\mathcal{L}_X \mathsf{Vol}_P = \mathsf{Div}(X)\,\mathsf{Vol}_P$. Note that the divergence operator does not depend on the chosen connection.  We can then write
$$\mathsf{Div}(X) = \frac{1}{\sqrt{|g_M|} \, t^{-1}} \tr\big(g_M^{-1} \mathcal{L}_X g_M \big)\,.$$
A direct consequence of Proposition \ref{prop:Killing} is the following. 
\begin{proposition}
Let $(P, g, \Phi)$ be a Carrollian $\R^\times$-bundle with a connection. If $\mathcal{L}_{\Delta_P}g = 0$, then every $X_\sV \in \ker(g)$ is divergenceless.
\end{proposition}
Returning to Killing vector fields, as shown earlier, if $\Delta_P$ is Killing, then all vertical vector fields are Killing (See Proposition \ref{prop:Killing}). Equipping $P$ with an $\R^\times$-connection, we have a splitting $X = X_\sH + X_\sV$, and so we can consistently consider horizontal Killing vector fields.  Note that the Killing equation $\mathcal{L}_X g = 0$ does not depend on the connection. Addressing the question of horizontal Killing vector fields, let $X_\sH$ be the horizontal component of a Killing vector field $X$, i.e., locally  $X_\sH = X^a D_a =  X^a\partial_a + X^aA_a\Delta_P$. The Lie derivative of $g$ is given by (see \eqref{eqn:LieDer})
$$\mathcal{L}_{X_\sH}g = \left(X^a\frac{\partial}{\partial x^a} + \left(\dot{x}^b \frac{\partial X^a}{\partial x^b} + \dot{\mathbf{t}} (\Delta_P X^a) \right)\frac{\partial}{\partial \dot{x}^a}\right)~ \dot{x}^c \dot{x}^d g_{dc}(x)\,,$$
where we have used Proposition \ref{prop:Killing}, i.e., $\mathcal{L}_{X^aA_a\Delta_P}g=0$.  We then observe that 
$$\mathcal{L}_{X_\sH}g =  \dot{x}^a\dot{x}^b X^c \frac{\partial g_{ba}}{\partial x^c} + 2\dot{x}^a\dot{x}^b \frac{\partial X^c}{\partial x^b}g_{cb} + 2 \dot{\mathbf{t}}\dot{x}^a(\Delta_P X^b)g_{ba} =0 \,.$$
Then, for this equation to hold, we require $(\Delta_P X^b)=0$ as $g_{ab}$ is invertible. Thus, any horizontal vector field $X_\sH$ must be independent of $t$. Stated invariantly, $X_\sH$ is of weight zero, i.e., $[\Delta_P, X_\sH]=0$. Such vector fields are projectable, and this is independent of the chosen connection (examining local expressions shows this). Thus, horizontal Killing vector fields project to vector fields $X_M \in \Vect(M)$ that are Killing for $g_M$. We have thus established the following.
\begin{proposition}\label{prop:KillAreProj}
Let $(P,g)$ be a Carrollian $\R^\times$-bundle such that $\Delta_P$ is Killing. Then Killing vector fields $X\in \Vect(P)$  that have a (non-zero) component outside of $\Sec(\sV P)$ are projectable vector fields whose projection $X_M \in \Vect(M)$ is a (non-zero) Killing vector field for $g_M$.
\end{proposition}
We then view this situation as being akin to gauge symmetry. Vertical vector fields $X_\sV$ can be viewed as generating ``gauge symmetries'' as translations along the fibres of $P$ do not change the degenerate metric. The ``genuine'' symmetries are generated by horizontal Killing vector fields (given an $\R^\times$-connection) and these project to ``genuine'' Killing vector fields of $(M,g_M)$. 
%
%%%%%%%%%%%%%%%%%%%%%
%
\subsection{Null Geodesics on Carrollian $\R^\times$-bundles}\label{subsec:NullGeos}
We restrict our attention to Carrollian $\R^\times$-bundles with a chosen $\R^\times$-connection $(P, g, \Phi)$, such that $M$ is $n$-dimensional, the degenerate metric has signature $(1,1,\cdots,1,0)$, which is in line with the fibre coordinate being the degenerate direction, and $\Delta_P$ is Killing. This is, of course, the standard situation in physics. We will consider a non-degenerate metric of the form $g^{(-)}_\omega := g - \omega^2$ on the total space $P$. The reason for this is that we now have a Lorentzian manifold $(P, g^{(-)}_\omega)$, and so we can consider null geodesics - this will be important for our suggestive language. We view null geodesics as a method to probe the geometry rather than attach a direct physical interpretation to them.  The metric is then of the form
$$g^{(-)}_\omega =  \dot{x}^a \dot{x}^b( g_{ba}(x) - A_bA_a(x) )  - 2 \dot{\mathbf{t}}\dot{x}^aA_a(x)  -\dot{\mathbf{t}}^2 \,.$$
As $\Delta_P$ is Killing, there is a conserved charge along geodesics. This is easily seen to be $Q_P := -\omega = -(\dot{\mathbf{t}} + \dot{x}^aA_a)$, we will refer to this as the \emph{Carroll charge}. This charge has the interpretation as the `Carrollian energy' as it is associated with flow along the fibres of $P$ generated by the Euler vector field $\Delta_P$.\par 
The geodesic equations are 
\begin{align}\label{eqn:SpaGeo}
&\frac{\rmd^2 x^a}{\rmd \lambda^2} + \Gamma^a_{bc}\frac{\rmd x^c}{\rmd \lambda}\frac{\rmd x^b}{\rmd \lambda} + \left( \frac{1}{t}\frac{\rmd t}{\rmd \lambda}+\frac{\rmd x^a}{\rmd \lambda}A_a\right)(g_M)^{ab}F_{bc}\frac{\rmd x^c}{\rmd \lambda} \\
\nonumber &+ \frac{1}{2}\left( (g_M)^{ab}A_b F_{cd}+ A_c (g_M)^{ab}F_{bd}+ A_d (g_M)^{ab}F_{bc}\right)\frac{\rmd x^d}{\rmd \lambda}\frac{\rmd x^c}{\rmd \lambda} =0\,,\\
& \frac{\rmd^2 t}{\rmd \lambda^2} + t\left( \frac{1}{2}\left( \frac{\partial A_a}{\partial x^b}+ \frac{\partial A_b}{\partial x^a} \right)\frac{\rmd x^b}{\rmd \lambda}\frac{\rmd x^a}{\rmd \lambda} + \frac{\rmd^2 x^a}{\rmd \lambda^2}A_a\right) - \frac{1}{t}\left( \frac{\rmd t}{\rmd \lambda}\right)^2 =0\,, \label{eqn:TemGeo}
\end{align}
where $\lambda$ is the affine parameter describing the geodesics.  Let us consider null geodesics. The condition for a path to be null implies that
\begin{equation}\label{eqn:NullCon}
\frac{\rmd x^a}{\rmd \lambda}\frac{\rmd x^b}{\rmd \lambda}(g_M)_{ba}- \left(\frac{1}{t} \frac{\rmd t}{\rmd \lambda} +\frac{\rmd x^a}{\rmd \lambda}A_a \right)^2=0\,,
\end{equation}
where the second term is the Carrollian charge evaluated along the geodesic. \emph{Carrollian photons} we define as particles that move along the null geodesics on $P$  with $Q_P =0$.  We remark that $Q_P =0$ means that the derivative of the worldline of a Carrollian photon is a horizontal vector field.  As $g_M$ has signature $(1,1,\cdots,1)$,  i.e., the metric is strictly positive, Carrollian photons must follow paths $\frac{\rmd x^a}{\rmd \lambda} =0$. As $Q_P =0$ along the null geodesic we have that $\frac{1}{t}\frac{\rmd t}{\rmd \lambda} =0$, and so $\frac{\rmd t}{\rmd \lambda}=0$. Note that the geodesic equations for Carrollian photons do not depend on the chosen $\R^\times$-connection and that the derivative of their worldline is the zero vector field. Carrollian photons are particles that are ``frozen'' at a point $p  \in P$, locally given by $(x^a(\lambda), t(\lambda)) = (x^a_0, t_0)$, for all $\lambda \in \R$, which is interpreted as the Carrollian photons ``internal time''. Using the interpretation of the Carrollian charge as an ``energy'', we see that having zero ``Carrollian energy'' means the particle is unable to move. \par 
If we insist that the derivative of a null geodesic on $P$ be vertical, then by definition $\frac{\rmd x^a}{\rmd \lambda}=0$. The null constraint \eqref{eqn:NullCon} now implies that  $\frac{\rmd t}{\rmd \lambda}=0$ and so $Q_P=0$ along such a null geodesic. Thus, we recover Carrollian photons - this is consistent, as the zero vector field is both horizontal and vertical.\par 
Particles that follow null geodesics on $P$ with $Q_P \neq 0$, we refer to as \emph{charged Carrollian photons}. From our discussion of Carrollian photons, we must have  $\frac{\rmd x^a}{\rmd \lambda}\neq 0$ for charged Carrollian photons; they carry ``Carrollian energy'' and so can move. Charged Carrollian photons are then tachyons when viewed on $M$. However, in general, their motion will depend on the gauge potential and its curvature.\par 
If we take the simplifying condition that $P = M \times \R^\times$ and equip it with the trivial connection, then the geodesics for charged Carrollian photons take the form
\begin{align*}
& \frac{\rmd^2 x^a}{\rmd \lambda^2} + \Gamma^a_{bc}\frac{\rmd x^c}{\rmd \lambda}\frac{\rmd x^b}{\rmd \lambda} =0  \,,\\
&  \frac{\rmd^2 t}{\rmd \lambda^2}  - \frac{1}{t}\left( \frac{\rmd t}{\rmd \lambda}\right)^2 =0\,,
\end{align*}
subject to the non-holonomic constraint $\langle X, X\rangle_{g_M} =  Q_P^2 = const$, where $X^a = \frac{\rmd x^a}{\rmd \lambda}$ is the velocity vector on $M$. The temporal geodesic can be directly solved as before; $t(\lambda) = t_0 \rme^{C \lambda}$. As, in this case, $Q_P = - \dot{t}\,t^{-1}$, we observe that $C = \pm Q_P$.
\begin{example}[Black Hole Horizon]
Following Example \ref{exa:BHHorizon}, we take $P = S^2 \times \R^\times$, and $g = (2GM)^2 \, g_{S^2}$, where $g_{S^2}$ is the standard round metric. We will furthermore consider the trivial connection. The spatial geodesics are great circles, as standard; however, the velocity along these circles for charged Carrollian photons is constrained to be constant. Using angular coordinates $(\vartheta, \varphi)$ on $S^2$ the non-holonomic constraint is 
$$Q_P^2 = (2GM)^2(\dot{\vartheta}^2 + \sin^2 \vartheta \, \dot{\varphi}^2 )\in \R\,.$$
Restricting to equatorial motion, i.e., $\vartheta = \pi/2$, the angular momentum becomes $\mathsf{L} = \pm (2GM) \, Q_P$. Then the equatorial motion is described by 
$$\varphi(\lambda) = \varphi_0 + \epsilon \frac{|Q_P|}{2 GM} \, \lambda\,, \qquad t(\lambda) = t_0 \rme^{ \delta \, |Q_P| \lambda}\,,$$
with $\epsilon, \delta = \pm 1$ independently. The signs $\epsilon$ and $\delta$ encode the direction of motion on the equator and the exponential growth or decay in time. Note that the equatorial motion is independent of $t^{-1}$, and so extends via the smooth inclusion $P \hookrightarrow L$  to include $t =0$. Assuming that $\lambda > 0$ and $\delta = -1$, a charged Carrollian photon moves in a circular orbit on the horizon at a constant speed, while the ``extrinsic clock'' winds down, i.e. $t \rightarrow 0$ as $\lambda \rightarrow \infty$. It will be illustrative to consider ``logarithmic time'' $u =  \ln |t|$ (see \cite{Duval:2014b}) , and so the temporal geodesic becomes $u(\lambda) = \ln |t_0| +  \delta |Q_P|\, \lambda$. By considering $u$ as the ``proper Carrollian time'', we can re-parametrise the spatial geodesic  as
$$\varphi(u) =  \tilde{\phi}_0 \pm \frac{1}{2 GM}\,u\,.$$
\end{example}
Using the null condition, the null geodesic equations can be written as 
\begin{align}\label{eqn:SpaGeo2}
&\frac{\rmd^2 x^a}{\rmd \lambda^2} + \Gamma^a_{bc}\frac{\rmd x^c}{\rmd \lambda}\frac{\rmd x^b}{\rmd \lambda} - Q_P \,(g_M)^{ab}F_{bc}\frac{\rmd x^c}{\rmd \lambda} + \frac{1}{2}\left( (g_M)^{ab}A_b F_{cd}+ A_c (g_M)^{ab}F_{bd}+ A_d (g_M)^{ab}F_{bc}\right)\frac{\rmd x^d}{\rmd \lambda}\frac{\rmd x^c}{\rmd \lambda} =0\,,\\
& \frac{\rmd t}{\rmd \lambda}= - t \left (Q_P + \frac{\rmd x^a}{\rmd \lambda} A_a \right)\,, \label{eqn:TemGeo2}
\end{align}
where we have the non-holonomic constraint $\langle X, X\rangle_{g_M} =  Q_P^2 = \mathrm{const}$, and noting that the first differential consequence of the first-order temporal equation \eqref{eqn:TemGeo2} is  \eqref{eqn:TemGeo}. Explicitly, using $Q_P \in \R$ for null geodesics,
\begin{align*}
\frac{\rmd^2 t}{\rmd \lambda^2}&= - \frac{\rmd t }{\rmd \lambda} \left(Q_P + \frac{\rmd x^a}{\rmd \lambda}A_a \right) - t \left( \frac{\rmd}{\rmd \lambda}Q_P + \frac{\rmd^2 x^a}{\rmd \lambda^2} + \half \left( \frac{\partial A_a}{\partial x^b}+\frac{\partial A_b}{\partial x^a} \right)\frac{\rmd x^b}{\rmd \lambda}\frac{\rmd x^b}{\rmd \lambda}\right)\\
&= - t \left( \half \left( \frac{\partial A_a}{\partial x^b}+\frac{\partial A_b}{\partial x^a} \right)\frac{\rmd x^b}{\rmd \lambda}\frac{\rmd x^b}{\rmd \lambda} + \frac{\rmd^2 x^a}{\rmd \lambda^2}A_a\right)- \frac{\rmd t}{\rmd \lambda}Q_P - \frac{\rmd t}{\rmd \lambda}\frac{\rmd x^a}{\rmd \lambda}A_a\\
&= - t \left( \half \left( \frac{\partial A_a}{\partial x^b}+\frac{\partial A_b}{\partial x^a} \right)\frac{\rmd x^b}{\rmd \lambda}\frac{\rmd x^b}{\rmd \lambda} + \frac{\rmd^2 x^a}{\rmd \lambda^2}A_a\right)+  \frac{1}{t}\left( \frac{\rmd t}{\rmd \lambda}\right)^2\,,
\end{align*}
which gives \eqref{eqn:TemGeo}. \par
As $Q_P \in \R$ for null geodesics, and in particular, non-zero for charged Carrollian photons, \eqref{eqn:TemGeo2} can formally be solved, i.e., we have 
\begin{equation}\label{eqn:SolTempGeo}
t(\lambda) = t_0 \rme^{ \delta\, |Q_P| \lambda} \, \rme^{\left( - \int  \rmd \lambda~\frac{\rmd x^a}{\rmd \lambda} A_a \right)}\,.
\end{equation}
From the non-holonomic constraint, we observe that $\frac{\rmd x^a}{\rmd \lambda}$ scales with $Q_P$. This observation is essential in exploring the limits of the null geodesics.  Let us assume $A_a \sim \varepsilon << 1$ (but non-zero), while $F_{ab} \sim 1$ (or larger). Notice that the first three terms of the spatial geodesic \eqref{eqn:SpaGeo2} scale like $Q_P^2$, while the fourth term scales like $\varepsilon Q_P^2$. Thus, for a ``small enough'' gauge field, we can safely drop the fourth term.  Similarly, in this limit, we drop $A_a$ from \eqref{eqn:SolTempGeo}. Using ``logarithmic time'' as the ``proper Carrollian time", in the ``small $A$'' limit the spatial geodesic equation \eqref{eqn:SpaGeo2} can be cast into the form
\begin{equation}\label{eqn:SpaGeo3}
\frac{\rmd^2 x^a}{\rmd u^2} + \Gamma^a_{bc}\frac{\rmd x^c}{\rmd u}\frac{\rmd x^b}{\rmd u} =  \sgn(Q_P) \,(g_M)^{ab}F_{bc}\frac{\rmd x^c}{\rmd u}\,.
\end{equation}
This is reminiscent of the (Euclidean) equation describing the motion of a charged particle on an $n$-dimensional manifold subject to the Lorentz force.\par
To understand this system, consider Riemann normal coordinates about $m \in M$, in which the spatial geodesic equation locally becomes
$$\frac{\rmd^2 x^a}{\rmd u^2}  =  \sgn(Q_P) \,\delta^{ab}F_{bc}(m)\frac{\rmd x^c}{\rmd u}\,,$$  
where we also make the approximation that $F_{ab}$ can be reasonably approximated by its value at $m$. Then, locally, i.e., on a small enough neighbourhood of $m \in M$, the motion in ``logarithmic time'' is helical or circular, depending on details of the initial conditions and the exact form of $F$. If $A =0$, then $F =0$, and the geodesics are (locally) straight lines. However, as soon as $A$ is non-zero, but still small, provided $F$ is non-zero and not small, the motion is drastically different, i.e., helical or circular. 
\begin{example}[Black Hole Horizon]
Returning to the example presented above, we now modify the situation to include a small $A$, but a non-zero $F$.  The $\varphi$-component of the spatial geodesics \eqref{eqn:SpaGeo3} have on the right-hand side 
$$\pm \,g_{S^2}^{\vartheta\vartheta} F_{\vartheta\varphi} \frac{\rmd \varphi}{\rmd u}\,.$$ 
Thus, if we initially restrict to equatorial motion ($\vartheta = \pi/2$), the Lorentz-force-like term directly contributes to the geodesic equation for $\vartheta$. This immediately shows that the initial equatorial motion will be `pushed off' the equator, leading to a qualitatively different trajectory compared to the great circles.
\end{example}
%
%%%%%%%%%%%%%%%%%%%%%
%
\section{Concluding Remarks}\label{sec:ConRem}
We have started to develop an approach to Carrollian geometry using principal $\R^\times$-bundles. The advantages of this approach are:
\begin{itemize}
\item The principal $\R^\times$-bundles can be non-trivial, but $\sV P$, which is identified with the kernel of the degenerate metric, is trivial. This simplifies the picture,  allowing us to work with a globally defined fundamental vector field even if the bundle is non-trivial.
\item We can employ $\R^\times$-connections in the theory rather than just general Ehresmann connections, and this allows a natural choice of affine connection on the principal $\R^\times$-bundle as the Levi-Civita connection of a non-degenerate metric constructed from the degenerate metric and the connection. The construction is tied to Kaluza--Klein geometry.
\item We have given a few simple examples that can be included in our formalism, such as the event horizon of a Schwarzschild black hole.
\end{itemize}
However, not all examples of weak Carrollian manifolds can be covered by the approach presented here; we require that the Carroll vector field is non-singular, complete, and that the associated foliation is simple and has non-compact leaves. This restriction still covers a large class of physically relevant examples. Relaxing the condition that the Carroll vector field be non-singular will introduce non-regular foliations to the theory.  For example, Ecker et al. (see \cite{Ecker:2023}) in their study of Carrollian analogues of black holes, have argued the necessity of weakening the non-singular condition. While such Carrollian manifolds are easy enough to define (see Definition \ref{def:CarrMan}), they cannot be directly tackled within the framework proposed in this paper.  Nonetheless, the differential geometry of Carrollian $\R^\times$-bundles is expected to be useful in physics beyond the standard examples. 
%
%%%%%%%%%%%%%%%%%%%%%%%%%%%
%
\section*{Acknowledgements}  
The author thanks Janusz Grabowski for his help and encouragement, and Daniel Grumiller for the questions that led to improvements in this paper.  The two anonymous referees are graciously acknowledged for their helpful comments.
%
%%%%%%%%%%%%%%%%%%%%%%%%%%%%%%%%%%%%%%%%
%

%%%%%%%%%%%%%%%%%%%%%%%%%%%%%%%%%%

\begin{thebibliography}{10}
\begin{small}

\bibitem{Bagchi:2025}
Bagchi, A., Banerjee, A., Dhivakar, P.,  Mondal, S. \&  Shukla, A.,
The Carrollian Kaleidoscope, \href{https://doi.org/10.48550/arXiv.2506.16164}{arXiv:2506.16164 [hep-th]}.

\bibitem{Blitz:2024}
Blitz S. \&  McNut D., 
Horizons that gyre and gimble: a differential characterization
of null hypersurfaces, \href{https://doi.org/10.1140/epjc/s10052-024-12919-y}{\emph{Eur. Phys. J. C}}  \textbf{84}, paper 561 (2024).

\bibitem{deBoer:2022}
de Boer, J.,  Hartong, J., Obers, N.A., Sybesma, W. \&  Vandoren, S.,
 Carroll Symmetry, Dark Energy and Inflation, \href{https://doi.org/10.3389/fphy.2022.810405}{\emph{Front. in Phys.}} \textbf{10},  810405 (2022).

\bibitem{Bruce:2017}
Bruce, A.J., Grabowska, K. \& Grabowski, J.,
Remarks on contact and Jacobi geometry,
\href{https://doi.org/10.3842/SIGMA.2017.059}{\emph{SIGMA}} \textbf{13}, Paper 059, 22 p. (2017). 

\bibitem{Bruce:2025a}
Bruce, A.J.,
Carrollian $\R^\times$-bundles II: Sigma Models on Event Horizons, \href{arXiv:2507.00544 [gr-qc]}{arXiv:2507.00544 [gr-qc]} (2025).

\bibitem{Bruce:2025b}
Bruce, A.J.,
Carrollian $\R^\times$-bundles III: The Hodge Star and Hodge--de Rham Laplacians, \href{https://doi.org/10.48550/arXiv.2507.21906}{arXiv:2507.21906 [math.DG]} (2025).

\bibitem{Cardona:2016}
Cardona, B., Gomis, J. \& Pons, J.M.,
Dynamics of Carroll strings,
\href{https://doi.org/10.1007/JHEP07(2016)050}{\emph{J. High Energy Phys.}} \textbf{7}, Paper No. 50, 11 p. (2016). 

\bibitem{Chandrasekaran:2022}
Chandrasekaran, V., Flanagan, É.É., Shehzad,I. \& Speranza, A.J.,
Brown-York charges at null boundaries,
\href{https://doi.org/10.1007/JHEP01(2022)029}{\emph{J. High Energy Phys.}} No. 1, Paper No. 29, 29 p. (2022). 

\bibitem{Ciambelli:2019}
Ciambelli, L.,  Leigh, R.G.,  Marteau, C. \& Petropoulos,  P.M.,
Carroll structures, null geometry, and conformal isometries,
\href{https://doi.org/10.1103/PhysRevD.100.046010}{\emph{Phys. Rev. D}} \textbf{100}, 046010 (2019).

\bibitem{Donnay:2019}
Donnay, L. \& Marteau, C.,
Carrollian physics at the black hole horizon,
\href{https://doi.org/10.1088/1361-6382/ab2fd5}{\emph{Class. Quantum Grav.}} \textbf{36}, No. 16, Article ID 165002, 19 p. (2019). 

\bibitem{Duval:2014a}
Duval, C., Gibbons, G.W. \& Horvathy, P.A.,
Conformal Carroll groups,
\href{https://doi.org/10.1088/1751-8113/47/33/335204}{\emph{J. Phys. A, Math. Theor.}} \textbf{47}, No. 33, Article ID 335204, 23 p. (2014). 

\bibitem{Duval:2014b}
Duval, C., Gibbons, G.W. \& Horvathy, P.A.,
Conformal Carroll groups and BMS symmetry,
\href{https://doi.org/10.1088/0264-9381/31/9/092001}{\emph{Class. Quantum Grav.}} \textbf{31}, 092001 (2014). 

\bibitem{Duval:2014}
Duval,C., Gibbons, G.W.,  Horvathy, P.A. \&  Zhang, P.M., Carroll versus Newton and Galilei: Two Dual Non-Einsteinian Concepts of Time, \href{https://doi.org/10.1088/0264-9381/31/8/085016}{\emph{Class. Quantum Grav.}} \textbf{31}, 085016 (2014).

\bibitem{Ecker:2023}
Ecker, F., Grumiller, D., Hartong, J.,  Perez, A., Prohazka, S. \& Troncoso, R.,
Carroll black holes, \href{https://doi.org/10.21468/SciPostPhys.15.6.245}{\emph{SciPost Phys.}} \textbf{15}, No. 6, Paper No. 245, 59 p. (2023). 

\bibitem{Grabowski:2013}
Grabowski, J.,
Graded contact manifolds and contact Courant algebroids,
\href{https://doi.org/10.1016/j.geomphys.2013.02.001}{\emph{J. Geom. Phys.}} \textbf{68}, 27--58 (2013). 

\bibitem{Heanneaux:1979}
Henneaux, M.,
Zero Hamiltonian signature spacetimes,
\emph{Bull. Soc. Math. Belg., Sér. A} \textbf{31}, 47--63 (1979). 

\bibitem{Kerner:2001}
Kerner, R., Martin, J., Mignemi, S. \&  van Holten J-W.,
Geodesic deviation in Kaluza-Klein theories,
\href{https://doi.org/10.1103/PhysRevD.63.027502}{\emph{Phys. Rev. D}} \textbf{63}, 027502 (2001)
 
\bibitem{Kolar:1993}
Kolář, I.,  Michor, P.W. \&  Slovák, J.,
\href{https://doi.org/10.1007/978-3-662-02950-3}{Natural operations in differential geometry}, \emph{Springer-Verlag, Berlin}, 1993, vi+434 pp. ISBN: 3-540-56235-4.

\bibitem{Lévy-Leblond:1965}
Lévy-Leblond, J.M.,
Une nouvelle limite non-relativistic du groupe de Poincaré,
\href{https://www.numdam.org/item/AIHPA_1965__3_1_1_0/}{\emph{Ann. Inst. Henri Poincaré, Nouv. Sér., Sect. A}} \textbf{3}, 1--12 (1965). 

\bibitem{Mars:2024}
Mars, M.,
Abstract null geometry, energy-momentum map
and applications to the constraint tensor,
\href{https://dx.doi.org/10.4310/BPAM.2024.v1.n2.a8}{\emph{Beijing J. of Pure and Appl. Math.}}
 \textbf{1}, Number 2, 797--852 (2024).

\bibitem{Moerdijk:2003}
Moerdijk, I. \& Mrčun, J.,
\href{https://doi.org/10.1017/CBO9780511615450}{Introduction to foliations and Lie groupoids}, Cambridge Studies in Advanced Mathematics, 91, \emph{Cambridge: Cambridge University Press}. ix, 173 p. (2003). 

\bibitem{SenGupta:1966}
Sen Gupta, N.D., On an analogue of the Galilei group, \href{https://doi.org/10.1007/BF02740871}{\emph{Nuovo Cimento A}} \textbf{44}, 512--517 (1966).
\end{small}
\end{thebibliography}
\end{document}